\documentclass[letterpaper]{amsart}
\usepackage[utf8]{inputenc}
\usepackage{graphicx} 

\usepackage{amsfonts}
\usepackage{mathtools}
\usepackage{amsmath, amsthm, amsopn, amssymb, microtype}
\usepackage{mathrsfs} 
\usepackage[all,cmtip]{xy}
\usepackage{enumerate, tikz, etoolbox, intcalc, caption, subcaption, afterpage}
\usepackage[english]{babel}
\usepackage{enumitem}
\usepackage{csquotes}

\usepackage[colorlinks=true,urlcolor=blue,pdfborder={0 0 0}]{hyperref}
\hypersetup{linkcolor=[rgb]{0.8,0,0}}
\hypersetup{citecolor=[rgb]{0,0.6,0}}
\usepackage[nameinlink]{cleveref}

\newtheorem{theorem}{Theorem}[section]
\newtheorem{definition}[theorem]{Definition}
\newtheorem{lemma}[theorem]{Lemma}
\newtheorem{corollary}[theorem]{Corollary}
\newtheorem{question}[theorem]{Question}

\newtheorem{proposition}[theorem]{Proposition}
\newtheorem{remark}[theorem]{Remark}

\usepackage[left=30mm, top=25mm, right=30mm, bottom=25mm]{geometry} 

\title{Characteristic Subgroup Growth}
\author{Liam Hanany and Alexander Lubotzky}
\thanks{The authors were supported by the European Research Council (ERC) under the European Union’s Horizon 2020 (N. 882751)}
\date{\today}

\begin{document}

\begin{abstract}
Let $s_n^\mathrm{ch}(\Gamma)$ denote the number of characteristic subgroups of index at most $n$ in a finitely generated group $\Gamma$. In response to a question of I. Rivin we show that if $\Gamma = F_r$ is the free group on $r \geq 2$ generators then the growth type of $s_n^{\mathrm{ch}}(F_r)$ is $n^{\mathrm{log}(n)}$. This is in contrast with the expectation of W. Thurston who predicted that there should be a difference between $r = 2$ and $r > 2$. Along the way we answer a question of \cite{barnea2020large} on the normal subgroup growth of large groups.
\end{abstract}

\maketitle
\begin{center}
\textit{Dedicated to the memory of Otto Kegel, an inspiring mathematician and human being.}
\end{center}

\section{Introduction}
Let $\Gamma$ be a finitely generated group. Denote by $a_n(\Gamma)$ the number of subgroups of index $n,$ and by $s_n(\Gamma)$ the number of subgroups of index at most $n$. Similarly, denote by $a_n^\lhd(\Gamma)$ the number of normal subgroup of index $n$, and by $s_n^\lhd(\Gamma)$ the number of normal subgroups of index at most $n$.

We use the convention of \cite[Section 0.1]{lubotzky2012subgroup} and say that a function $g(n)$ has growth type $f(n)$ if there are constants $a,b > 0$ such that $g(n) \leq f(n)^a$ for every $n$ and $g(n) \geq f(n)^b$ for infinitely many choices of $n$. We say that a group $\Gamma$ has subgroup growth of type $f(n)$ if $s_n(\Gamma)$ has growth type $f(n)$, and $\Gamma$ has normal subgroup growth of type $f(n)$ if $s_n^\lhd(\Gamma)$ has growth type $f(n)$.

It has been known for a long time that for $F = F_r$ the free group on $r$ generators that $a_n(F_r) \sim n \cdot n!^{r-1}$, see \cite{hall1949subgroups} as well as \cite[Chapter 2]{lubotzky2012subgroup} and the references therein.
A combination of the results of the second author \cite{lubotzky2001enumerating} and Mann \cite{mann1998enumerating} shows that $F_r$ has normal subgroup growth of type $n^{\mathrm{log}(n)}$. In fact, they showed more, proving that $s_n^\lhd(F_r) = n^{\Theta(r\log n)},$ where the implicit constants in the $\Theta$-notation are independent of both $n, r.$

In 2011, Igor Rivin \cite{Thurston} raised the question of the characteristic subgroup growth of the free group. That is, let $a_n^\mathrm{ch}(\Gamma)$ denote the number of characteristic subgroups of index $n$ and let $s_n^\mathrm{ch}(\Gamma)$ be the number of characteristic subgroups of index at most $n$.

He asked what is the growth of $s_n^\mathrm{ch}(F_r)$ for the free group of rank $r$. In response, William Thurston \cite{Thurston} predicted that the answer would be different between $r=2$ and $r\geq 3$, and that for $r\geq 3$ the growth will be dominated by a polynomial, i.e. of type $n$.

The initial trigger to our work came from the observation that the result in \cite{chen2025finite} implies that the characteristic growth for $F_2$ is of type $n^{\mathrm{log}(n)}$, which is maximal possible as it gives a lower bound for the normal subgroup growth. The main observation was that the representation used in \cite{chen2025finite} combined with \cite[Theorem B part ii]{larsen2004g2} gives this characteristic growth type. See Section \ref{Final section} for more details.

Admittedly, starting this project our intuition was similar to that of Thurston, without knowing his prediction, and we also assumed that there would be a difference between $r = 2$ and $r \geq 3$. To our surprise, and against Thurston's initial intuition as well as our own, we prove the following
\begin{theorem}
\label{Characteristic Growth of Free Groups}
For every $r \geq 2$, the characteristic subgroup growth of $F_r$ has type $n^{\mathrm{log}(n)}$.
\end{theorem}
Equivalently stated, this means that on a logarithmic scale, a positive proportion of $r$-generated groups up to a given cardinality $n$ are in fact a characteristic quotient of $F_r.$ For such a characteristic quotient, every Nielsen transformation on the $r$ generators yields an equivalent $r$-generating set up to an automorphism of the finite group.

In fact our theorem is even more general. Recall that a finitely generated group is called large if there is a finite index subgroup with a surjective map to a non-abelian free group.
\begin{theorem}
\label{Characteristic Growth of Large Groups}
For every finitely generated $\Gamma$, if $\Gamma$ is large then the characteristic subgroup growth of $\Gamma$ is of type $n^{\mathrm{log}(n)}$.
\end{theorem}
This answers \cite[Problem 7]{barnea2020large} in a strong way. They ask if the normal subgroup growth of a large group is of type $n^{\mathrm{log}(n)}.$ This question looks difficult, as there is no reason apriori that the normal subgroup growth type will be a comensurability invariant. Indeed, if $\Delta \leq \Gamma$ is of finite index, the fact that $\Delta$ has large normal subgroup growth does not immediately imply large normal subgroup growth for $\Gamma$, as normal subgroups of $\Delta$ are no longer necessarily normal in $\Gamma$, even if we assume without loss of generality that $\Delta$ is characteristic in $\Gamma$. However, using our approach, characteristic subgroups of $\Delta$ will in fact give us normal (and even characteristic) subgroups of $\Gamma$.

The reader is encouraged to look at the paper \cite{barnea2020large} that seems to be the only work so far which has dealt with characteristic subgroup growth. Our work uses in a crucial way Lemma 9 from that paper.

The results in \cite{barnea2020large} and Theorems 1,2 above may suggest that the normal subgroup growth and the characteristic subgroup growth are always the same.
However, this is not true, and we construct a counter-example in Section \ref{Section on Segal's groups}, using a construction given by Segal \cite{segal2000finite}.
\begin{proposition}
\label{A group with larger normal subgroup growth}
There exists a finitely generated group $\Gamma$ with normal subgroup growth and characteristic subgroup growth of different types.
\end{proposition}

\section{Zassenhaus filtration}
The main tool for our proof will be the Zassenhaus filtration, sometimes also called the Dimension subgroups, see \cite[Chapters 11 and 12]{dixon2003analytic}, \cite{minavc2016dimensions}.
It is defined as follows, see \cite[Definition 11.1, Theorem 11.2]{dixon2003analytic}:
\begin{definition}
Let $\gamma_n(\Gamma)$ be the lower central series of $\Gamma$. Define $D_n(\Gamma) = \prod_{ip^j \geq n} \gamma_i(\Gamma)^{p^j}$.
Equivalently, this sequence can be defined inductively by $D_1(\Gamma) = \Gamma$ and $$D_n(\Gamma) = D_{\lceil\frac{n}{p}\rceil}(\Gamma)^p \prod_{j + k = n}[D_j(\Gamma), D_k(\Gamma)].$$
\end{definition}
In particular $D_2(\Gamma) = \Gamma^p[\Gamma, \Gamma]$.

The following Lemma is crucial in our proof of Theorems \ref{Characteristic Growth of Free Groups}, \ref{Characteristic Growth of Large Groups}. It implies that the action of $\mathrm{Aut}(\Gamma)$ on every subquotient of the Zassenhaus filtration factors through the fixed finite group $\mathrm{Aut}(\Gamma/\Gamma^p[\Gamma,\Gamma]).$
\begin{lemma}
\label{Automorphism Action}
Let $\varphi:\Gamma \to \Gamma$ be an automorphism of $\Gamma$ acting trivially on the $p$-elementary abelianization $\Gamma/\Gamma^p[\Gamma,\Gamma] = D_1 / D_2$ of $\Gamma$. Then $\varphi$ acts trivially on the subquotients of the Zassenhaus filtration $D_n / D_{n + 1}$.
\end{lemma}
\begin{proof}
We prove the claim by induction on $n$. For $n = 1$ it is the hypothesis.

Firstly, fix two elements $x \in D_m, y \in D_k$, for $m,k < n$. Then $\varphi(x) = xa$ and $\varphi(y) = yb$ for some $a \in D_{m + 1}, b \in D_{k + 1}$. Hence,
    $$\varphi([x, y]) = [\varphi(x),\varphi(y)] = [xa, yb] = xayba^{-1}x^{-1}b^{-1}y^{-1} = x[a,y]x^{-1} \cdot xy [a, b] x^{-1}y^{-1} \cdot y[x, b]y^{-1}$$
Note that $[a,y] \in [D_{m+1}, D_k] \leq D_{m + k + 1}$ and $[x, b] \in [D_m, D_{k + 1}] \leq D_{m + k + 1}$. Moreover $$xy[a,b]x^{-1}y^{-1} = [x,y] yx[a,b]x^{-1}y^{-1},$$and therefore $\varphi([x,y]) \in [x,y] D_{m + k + 1}$.

Now, fix $x \in D_m$, for $m < n$. Then $\varphi(x) = xa$ for $a \in D_{m + 1}$, and
$$\varphi(x^p) = \varphi(x)^p = (xa)^p \equiv x^p \ \mathrm{mod} D_{mp+1},$$ by \cite[Corollary 11.10 part ii]{dixon2003analytic}, which is proved using the Hall-Petresco formula.

Combining these two computations, we obtain using the inductive definition for the Zassenhaus filtration that for every $x \in D_n$, $\varphi(x) \in xD_{n + 1}$, as needed.
\end{proof}
We will also use the following two lemmas \cite[Lemma 9, Lemma 10]{barnea2020large}:
\begin{lemma}
    Let $G$ be a finite group and $p$ a prime number. Then there is a constant $c > 0$ such that for all $\mathbb{F}_p[G]$-modules $M$ of $\mathbb{F}_p$-dimension $d$, there exist submodules $M_1 \leq M_2 \leq M$ such that $M_2 / M_1$ is the direct sum of at least $cd$ isomorphic simple modules.
\end{lemma}
\begin{lemma}
    Let $G$ be a finite group and $M$ a finite dimensional $\mathbb{F}_p[G]$-module. Then there is some $c>0$ such that $M^n$ contains at least $e^{cn^2}$ submodules.
\end{lemma}
We summarize these in the following Corollary,
\begin{corollary}
\label{Lots of Submodules}
Fix a finite group $G$ and prime $p$. Then there is a constant $c > 0$ such that every $\mathbb{F}_p[G]$-module of $\mathbb{F}_p$-dimension $d$ has at least $e^{cd^2}$-submodules.
\end{corollary}
The last ingredient needed for the proof of our main results is a computation of the dimensions of the subquotients in the Zassenhaus filtration of the free group \cite[Proposition 3.4]{minavc2016dimensions} (see also \cite{lazard1954groupes} for a connection to the free reduced Lie algebra):
\begin{proposition}
\label{Zassenhaus dimension computation}
Let $\Gamma = F_r$ be the free group on $r$ generators.
Define
$$w_n(F_r) = \frac{1}{n} \sum_{m | n} \mu(m) r^{\frac{n}{m}}$$
for $\mu$ the Möbius function. Moreover, if $n = p^k m$ for $p \nmid m$ define
$$c_n(F_r) = \sum_{i=0}^kw_{p^im}(F_r)$$
Then $D_n(F_r)/D_{n+1}(F_r)$ is a vector space of dimension $c_n(F_r)$ over $\mathbb{F}_p$.
\end{proposition}
Since the growth of the sequence $c_n$ is exponential, we can deduce the following Corollary. The constants appearing in this Corollary have not been optimized, and arbitrary constants would be sufficient in proving the main results.
\begin{corollary}
\label{exponential growth}
Let $G = F_r$ be the free group on $r \geq 2$ generators. Then for sufficiently large $n$,
$$\mathrm{dim}_{\mathbb{F}_p}(D_n(F_r) / D_{n+1}(F_r)) \geq \frac{1}{2n} \cdot r^n$$ and
$$[D_n(F_r): D_{n+1}(F_r)] \geq [G: D_{n + 1}(F_r)]^{\frac{1}{15}}$$
\end{corollary}
\begin{proof}
For the first claim it is sufficient to notice that $$c_n(F_r) \geq w_n(F_r) \geq \frac{1}{n} \cdot (r^n - \sum_{k=0}^{\lfloor\frac{n}{2}\rfloor} r^k) \geq \frac{1}{n}(r^n - r^{\lfloor\frac{n}{2}\rfloor + 1})$$
which is greater than $\frac{1}{2n} \cdot r^n$ for sufficiently large $n$.
For the second claim, we start by giving an upper bound on $w_n(F_r)$
$$w_n(F_r) \leq \frac{1}{n} \cdot (r^n + \sum_{k=0}^{\lfloor\frac{n}{2}\rfloor} r^k) \leq \frac{2}{n} \cdot r^n$$
for every $n \geq 1$. We now claim that $c_n(F_r) < \frac{3}{n} \cdot r^n$ for sufficiently large $n.$ Indeed, $$c_n(F_r) \leq w_n(F_r) + \sum_{k=1}^{\lfloor\frac{n}{2}\rfloor} w_k(F_r)\leq \frac{2}{n} \cdot r^n + 2\cdot r^{\lfloor\frac{n}{2}\rfloor + 1} < \frac{3}{n} \cdot r^n$$
for sufficiently large $n,$ say $n \geq n_0$. Denote $M = \sum_{k=1}^{n_0 - 1}c_k(F_r)$.

We can now bound $$\sum_{k = 1}^{n - 1} c_k(F_r) \leq M + 3 \cdot \sum_{k=n_0}^{n - 1} \frac{r^k}{k} \leq M + 3 \cdot (r^{\lfloor \frac{n}{2} \rfloor + 1} + \frac{1}{\frac{n}{2}} \cdot r^n) < \frac{7}{n}  \cdot r^n < 14 \cdot c_n(F_r)$$ for sufficiently large $n.$ This implies that $c_n(F_r) > \frac{1}{15} \cdot \sum_{k=1}^n c_k(F_r)$ as needed.
\end{proof}
\section{Proofs of the main results}
Using the results from the previous section we are ready to prove Theorem \ref{Characteristic Growth of Free Groups}.
\begin{proof}[Proof of Theorem \ref{Characteristic Growth of Free Groups}]
Let $F_r$ be a free group of rank $r$. Let $D_n$ be its Zassenhaus-filtration. This is a descending sequence of characteristic subgroups of $F_r$.

$D_1 / D_2 \simeq (\mathbb{Z}/p)^r$ is an elementary $p$-group. By Lemma \ref{Automorphism Action}, any automorphism acting trivially on $D_1 / D_2$ also acts trivially on each of the subquotients of the Zassenhaus filtration.
The image of $\mathrm{Aut}(F_r)$ in $\mathrm{Aut}(D_1/D_2)$ is a finite group, specifically the group $\mathrm{SL}^{\pm1}_r(\mathbb{F}_p)$ of $r$-by-$r$ matrices in $\mathbb{F}_p$ of determinant $\pm 1$. Therefore, the action of $\mathrm{Aut}(F_r)$ on $D_n/D_{n+1}$ factors through this fixed finite group.

Denote by $N = [G : D_{n + 1}]$. Then $\mathrm{dim}(D_n / D_{n + 1}) \geq \frac{1}{15}\mathrm{log}_p(N)$ by Corollary \ref{exponential growth}.
Hence by Corollary \ref{Lots of Submodules}, we get that there are at least $e^{\frac{c}{225}\mathrm{log}_p(N)^2} = e^{\Omega_{p, r}(\mathrm{log}(n)^2)}$-submodules for this $\mathrm{SL}_r^{\pm 1}(\mathbb{F}_p)$-module. Each such submodule can be lifted and yields a subgroup $D_{n + 1} \leq H \leq D_n$. As this is a submodule for the $\mathrm{Aut}(F_r)$-action, the subgroup $H$ is characteristic in $F_r.$ Moreover, the index of $H$ is bounded by $N$. This proves the claim.
\end{proof}
We now show how these methods in fact imply Theorem \ref{Characteristic Growth of Large Groups}.
\begin{proof}[Proof of Theorem \ref{Characteristic Growth of Large Groups}]
Let $\Gamma$ be a large group, and let $\Gamma_0 \leq \Gamma$ be a finite index subgroup with a surjective map to a non-abelian free group. By moving to a finite index subgroup, we may assume that $\Gamma_0 \leq \Gamma$ is a characteristic finite index subgroup. $\Gamma_0$ still surjects onto a finite index subgroup of the non-abelian free group, i.e. a non-abelian free group. This subgroup is clearly finitely generated as it is the homomorphic image of the finitely generated $\Gamma_0$. Denote the corresponding surjection by $\varphi: \Gamma_0 \to F_r$. We may also assume up to a further finite index subgroup that $r \geq 3$.

Consider $D_n(\Gamma_0)$ the Zassenhaus filtration of $\Gamma_0$. Since $\varphi$ is surjective, $\varphi(D_n(\Gamma_0)) = D_n(F_r)$. In particular, the map $\varphi$ induces a surjection $D_n(\Gamma_0) / D_{n + 1}(\Gamma_0) \twoheadrightarrow D_n(F_r) / D_{n + 1}(F_r)$, and hence we get a lower bound on the dimensions $c_n \coloneq \mathrm{dim}(D_n(\Gamma_0) / D_{n + 1}(\Gamma_0)) \geq \frac{1}{2n} \cdot r^n$.

We now know that the dimension sequence $c_n$ increases exponentially, with rate strictly greater than $2$. It follows that for infinitely many $n$, $$\sum_{m<n} c_m \leq c_n.$$
Indeed, otherwise there is some $n_0$ such that for every $n \geq n_0$ $c_n < \sum_{m < n} c_m$ and in particular $$c_n < 2^{n - n_0} \cdot \sum_{m = 1}^{n_0} c_m \leq 2^n \cdot M$$ for some constant $M > 0$. This is impossible as $r > 2$.

For every such $n$, it follows that $\mathrm{dim}(D_n(\Gamma_0) / D_{n+1}(\Gamma_0)) \geq \frac{1}{2} \mathrm{log}_p([\Gamma_0 : D_{n + 1}(\Gamma_0)])$. Arguing as in the proof of Theorem \ref{Characteristic Growth of Free Groups}, we get for $N = [\Gamma_0: D_{n + 1}(\Gamma_0)]$ at least $e^{c \mathrm{log}_p(N)^2}$ characteristic subgroups of $\Gamma_0$ of index $\leq N$. This implies the claim.
\end{proof}

\begin{remark}
We remark that our proof further implies that there is a constant $c = c(\Gamma)$ such that for every $N$
$$s_N^{\mathrm{ch}}(\Gamma) \geq N^{c\cdot\mathrm{log}(N)}.$$

Indeed, the proof constructs recursively a sequence of pairs of integers $n_k, N_k$ such that
$$N_k = [\Gamma_0:D_{n_k + 1}(\Gamma_0)] = e^{\Theta_{\Gamma}(n_k)}$$ and
$$s_{N_k}^{\mathrm{ch}}(\Gamma) \geq N_k^{c \cdot \mathrm{log}(N_k)} = e^{\Theta_{\Gamma}(n_k^2)}.$$
It suffices to show that the rate of growth of the sequence $(n_k)$ is at most exponential, i.e. $n_{k + 1} = O_\Gamma(n_k),$ and then the achieved lower bounds corresponding to $k, k + 1$    will differ at most by a constant power, proving a uniform bound for every $N$.

The sequence $(n_k)$ can be chosen as those natural numbers such that $\sum_{m < n_k} c_m \leq c_{n_k}.$ Therefore, for every $n_k < n < n_{k + 1}$, we have the opposite inequality $\sum_{m < n} c_m > c_n$. As in the proof of the Theorem,
$$c_n < 2^{n - n_k} \cdot \sum_{m = 1}^{n_k} c_m < 2^{n - n_k} \cdot e^{O_\Gamma(n_k)}.$$
Since $c_n \geq \frac{r^n}{2n}$ for $r \geq 3$, this inequality will be violated for $n = O_\Gamma(n_k)$, as needed.
\end{remark}

\section{A group with different characteristic and normal subgroup growth types}
\label{Section on Segal's groups}
We will use the ideas of \cite{segal2000finite} in order to construct such a group. In this section let $\mathrm{Sym}(n)$ denote the symmetric group of permutations of $n$ points.

Firstly, the following lemma will be useful.
\begin{lemma}
\label{Normal subgroups of Wreath Products}
    Let $H \leq \mathrm{Sym}(n), K \leq \mathrm{Sym}(m)$ be transitive permutation groups, and let $G = H \wr K = H^m \rtimes K \leq \mathrm{Sym}(n \cdot m)$ be the permutational wreath product. Assume that $H \simeq S^q$ for some non-abelian finite simple group $S$, and $q \geq 1$.
    Then every quotient $\phi: G \twoheadrightarrow Q$ either factors through the quotient $G \twoheadrightarrow K$, or has kernel $N^m$ for some normal subgroup $N \lhd H$.
\end{lemma}
\begin{proof}
    Assume that the map does not factor through $K$. In such a case, the image of $H^m$ is non-trivial, and is therefore a direct product of copies of $S$. Consider the kernel $M = \mathrm{ker}(\phi|_{H^m}) = \mathrm{ker}(\phi) \cap H^m \lhd G$. Then $M \neq H^m$.
    
    As $M$ is $K$-invariant, it follows that it has the form $N^m$ for some normal subgroup $N \lhd H$, and $N \neq H$. Indeed, it is the direct product of a subset of the copies of $S$, and this subset is $K$-invariant, for the transitive permutation group $K$.

    Assume by contradiction that there is an element $x \in \mathrm{ker}(\phi)$ whose image in $K$ is non-trivial, i.e. that is not contained in $H^m$. Then the conjugation action of $\phi(x)$ on $Q$ is trivial. However, this action permutes the copies of $H$ non-trivially, and therefore acts non-trivially on $(H / N)^m \hookrightarrow Q$, in contradiction.
\end{proof}
This lemma gives us control over normal subgroups of iterated wreath products of powers of non-abelian simple groups. We describe this precisely in the following corollary
\begin{corollary}
\label{Normal subgroup count for iterated wreath products}
    Let $(H_k)_{k=1}^n$ be finite transitive permutation groups, $H_k \leq \mathrm{Sym}(\ell_k)$. Assume that each of the $H_k$ is isomorphic to a power of a non-abelian finite simple group $S_k$, $H_k \simeq S_k^{m_k}$. Denote $W_0 = 1$, and define recursively $W_k = H_k \wr W_{k - 1} \leq \mathrm{Sym}(\ell_1 \cdots \ell_k)$ as the permutational wreath product. Then the number of normal subgroups of $W_n$ is $$1 + \sum_{k=1}^n (2^{m_k} - 1).$$
    Moreover, the normal subgroups of index less than $|W_{n - 1}|$ of $W_n$ are precisely the lifts of normal subgroups of under the quotient map $W_n \twoheadrightarrow W_{n - 1}$.
\end{corollary}
\begin{proof}
    Denote $\ell = \ell_1 \cdots \ell_{n - 1}$, such that $W_{n - 1} \leq \mathrm{Sym}(\ell)$ and let $M \lhd W_n$.
    If $M$ contains $H_n^\ell$, then $M$ is induced by a normal subgroup of the quotient $W_n \twoheadrightarrow W_{n-1}$.
    Otherwise, by the previous Lemma, $M = N^\ell$ for some normal subgroup $N \lhd H_n$, and $N \neq H_n$, in which case it is a direct product of a proper subset of the $S_n$ factors.

    The first claim now follows by induction on $n$, where the $1$ comes from the base case $n = 0$.

    The second claim follows directly from Lemma \ref{Normal subgroups of Wreath Products}.
\end{proof}

We can now build our required group using the following theorem of Segal \cite[Theorem 5]{segal2000finite}, based on an earlier construction of Grigorchuk \cite{grigorchuk2000just}.
\begin{theorem}
\label{Segal dense subgroup}
    Let $(H_n)_{n=1}^\infty$ be a sequence of finite transitive permutation groups $H_n \leq \mathrm{Sym}(\ell_n)$.
    Assume that for every $n$ the permutation group $H_n$ has distinct point stabilizers.
    
    Define $(W_n)_{n=1}^\infty$ by $W_0 = 1$ and $W_n = H_n \wr W_{n - 1} \leq \mathrm{Sym}(\ell_1 \cdots \ell_n)$ the permutational wreath product. Denote by $W$ the profinite group $W = \varprojlim W_n$.
    Assume further that there is a finitely generated perfect group $\Lambda$ such that $H_n$ is an epimorphic image of $\Lambda$ for every $n \geq 1$.
    
    Then there is a dense finitely generated subgroup $\Gamma \leq W$ with the congruence subgroup property, i.e. such that the natural map from the profinite completion, $\hat{\Gamma} \to W,$ is an isomorphism.
\end{theorem}
We will need something more about the construction of this dense subgroup $\Gamma$ that follows directly from the method of proof of Theorem \ref{Segal dense subgroup}, which we describe in the following Proposition:
\begin{proposition}
\label{Inductive construction segal}
    The group $\Gamma$ constructed in Theorem \ref{Segal dense subgroup} satisfies two further properties:
    \begin{enumerate}
        \item $\varinjlim W_n \leq \Gamma$.
        \item For every $n$ there is a subgroup $\Gamma_n \leq \Gamma$ such that $\Gamma \simeq \Gamma_n \wr W_n = \Gamma_n^{\ell_1 \cdots \ell_n} \rtimes W_n$, the permutational wreath product.
    \end{enumerate}
\end{proposition}
\begin{proof}
    We follow the notation of \cite{segal2000finite}. In this notation $\Gamma$ is denoted by $A$ and $H_n$ is denoted by $S_n$, for $n \geq 0$ instead of $n \geq 1$.

    The group $A$ is constructed as a subgroup of $W$ which is thought of as the automorphism group of an infinite rooted tree, the degrees of the vertices at level $n$ given by the size of the underlying set of the permutation group $S_n$.

    In this notation, \cite[Lemma 1]{segal2000finite} implies that $A$ is the wreath product $\mathrm{top}_A(v) \wr S_0$, where $v$ is an arbitrary vertex at level $1$ of the tree, and the subgroup $\mathrm{top}_A(v)$ is the subgroup of automorphism acting trivially on every vertex outside the subtree of $v$. More specifically this follows from the fact that $\mathrm{rist}_A(1) = \mathrm{st}_A(1)$, where $\mathrm{st}_A(1)$ is the subgroup preserving the first level of the action on the tree, and $\mathrm{rist}_A(1) = \prod_{v \in [\ell_0]} \mathrm{top}_A(v)$, the product taken over vertices $v$ at the first level of the tree, and the fact that $S_0 \leq A$.

    The two claims now follow by the recursive definition of the group $A$, see the proof of \cite[Theorem 5]{segal2000finite}.
\end{proof}
In order to use Theorem \ref{Segal dense subgroup}, we need the following lemma
\begin{lemma}
\label{Suslin group}
    There is a finitely generated perfect group $\Lambda$ with epimorphisms to $\mathrm{PSL}_3(\mathbb{F}_p)^p$ for every prime $p$.
\end{lemma}
\begin{proof}
    Let $\Lambda = \mathrm{SL}_3(\mathbb{Z}[t])$, for an indeterminate $t$. Suslin proved \cite{suslin1977structure} that $\Lambda$ is perfect and finitely generated. Note that $\Lambda$ contains every upper and lower unipotent matrix $$\begin{pmatrix}1 & A(t) & B(t)\\ 0 & 1 & C(t) \\ 0 & 0 & 1\end{pmatrix}, \begin{pmatrix} 1 & 0 & 0\\ A(t) & 1 & 0 \\ B(t) & C(t) & 1\end{pmatrix}$$
    for any three polynomials $A, B, C \in \mathbb{Z}[t]$. It follows that any specialization to the finite field $\mathbb{F}_p$ given by substituting $t = a$ for some $a \in \mathbb{F}_p$ will have surjective image to both the upper and lower unipotents, and as these generate $\mathrm{SL}_3(\mathbb{F}_p)$, will have surjective image to $\mathrm{SL}_3(\mathbb{F}_p)$.

    Also, note that no two of the induced maps to $\mathrm{PSL}_3(\mathbb{F}_p)$ are the same (even up to an automorphism of $\mathrm{PSL}_3(\mathbb{F}_p)$), as we may find an upper unipotent that is trivial in one map and non-trivial in the other.

    We now use the classical fact, sometimes known as the Chinese Remainder Theorem for finite simple groups, that if a tuple of maps $\Lambda \to S = \mathrm{PSL}_3(\mathbb{F}_p)$ is surjective, and no two are equivalent up to an automorphism of $S$, then the joint map to $S^p$ is also surjective.
\end{proof}
We are now ready to prove Proposition \ref{A group with larger normal subgroup growth}.
\begin{proof}[Proof of Proposition \ref{A group with larger normal subgroup growth}]
Let $p_k$ denote the $k^{\mathrm{th}}$ prime number. Let $S_k = \mathrm{PSL}_3(\mathbb{F}_{p_k})$. $S_k$ is naturally a transitive permutation group of the projective plane $\mathbb{P}^2_{p_k}$, i.e. $S_k \leq \mathrm{Sym}(p_k^2 + p_k + 1).$ Note that this permutation action has distinct point stabilizers.

Let $H_k = S_k^{p_k} \leq \mathrm{Sym}((p_k^2 + p_k + 1)^{p_k})$, be the permutation group acting on tuples of $p_k$ points in the projective planes. Once again, this permutation action is transitive, and has distinct point stabilizers. Note also that $\mathrm{Sym}(p_k)$ has a natural action on this set of tuples, by permuting its coordinates. This yields the permutation group $S_k \wr \mathrm{Sym}(p_k) = H_k \rtimes \mathrm{Sym}(p_k) \leq \mathrm{Sym}((p_k^2 + p_k + 1)^{p_k}).$

Now, define $W_k, W$ as in Theorem \ref{Segal dense subgroup}, $W_k = H_k \wr W_{k-1}$, and $W = \varprojlim W_k$. Then using Theorem \ref{Segal dense subgroup} and Lemma \ref{Suslin group}, we obtain a finitely generated $\Gamma$ with profinite completion $\hat{\Gamma} \simeq W$.

By Corollary \ref{Normal subgroup count for iterated wreath products}, the number of normal subgroups of $\Gamma$ of index less than $|W_n|$ is $$1 + \sum_{k=1}^n (2^{p_k} - 1) > 2^{p_n} > 2^n.$$
Noting that the groups $H_n$ are all characteristically simple, and using Lemma \ref{Normal subgroups of Wreath Products} the number of characteristic open subgroups of $W$ of index less than $|W_n|$ is precisely $n + 1$, corresponding to the kernels of the maps $W \to W_k, k \leq n$. We wish to prove that these are precisely the characteristic subgroups of $\Gamma$ as well, and this will finish the proof. For this, we will need to lift sufficiently many automorphisms of $W_n$ to automorphisms of $\Gamma$. This will be done using Proposition \ref{Inductive construction segal}.

We have seen that $H_n$ has $\mathrm{Sym}(p_n)$ as outer automorphism group, and these automorphisms also naturally act on the underlying set of the permutation group $H_n$. These automorphisms can be lifted to automorphisms of $W_n$ by acting trivially on $W_{n-1}$ and acting on each of the copies of $H_n$ simultaneously. These yield automorphisms of $W_n$ that also act naturally on the underlying set of the permutation group $W_n$.

Now, using Proposition \ref{Inductive construction segal}, as $\Gamma$ is a permutational wreath product $\Gamma \simeq \Gamma_n \wr W_n$, these automorphisms of $W_n$ can be extended to an automorphism of $\Gamma$, where the action of the automorphism on the copies of $\Gamma_n$ are by permuting the coordinates. A normal finite index subgroup of $\Gamma$ preserved by all of these automorphisms must be the kernel of one of the maps $\Gamma \to W_n$, by Lemma \ref{Normal subgroups of Wreath Products}.
\end{proof}
\section{Remarks and Suggestions for future research}
\label{Final section}
\subsection{Alternative Proof for $\mathrm{Aut}(F_2)$}
By a result of Dyer, Formanek and Grossman \cite{dyer1982linearity}, the braid group $B_4$ and the automorphism group $\mathrm{Aut}(F_2)$ are related. More specifically, $B_4 / Z(B_4) \simeq \mathrm{Aut}^+(F_2)$, where $\mathrm{Aut}^+(F_2) \leq \mathrm{Aut}(F_2)$ is the index-$2$ subgroup acting by orientation preserving linear transformations of the abelianization $\mathbb{Z}^2$ of $F_2$. This isomorphism is induced by the conjugation action of $B_4$ on the subgroup $F_2 \simeq \langle \sigma_1\sigma_3^{-1}, \sigma_2\sigma_1\sigma_3^{-1}\sigma_2^{-1} \rangle \lhd B_4$.

This fact was used by Chen, Tiep and the second author \cite{chen2025finite} in order to generate characteristic simple quotients of $F_2$ of type $\mathrm{PSL}_3(\mathbb{F}_q), \mathrm{PSU}_3(\mathbb{F}_q)$ for all but finitely many (explicitly computed) $q$. These quotients were obtained by specializing the Burau representation $B_4 \to \mathrm{PGL}_3(\mathbb{Z}[q, q^{-1}])$, noting that the image of $F_2$ under this representation is contained in $\mathrm{PSL}_3(\mathbb{Z}[q, q^{-1}]).$ See \cite{hanany2025} for a far reaching generalization, however for our purposes the original ideas of \cite{chen2025finite} are sufficient.

It was shown in \cite{chen2025finite} that there are infinitely many specializations of the Burau representation of $F_2$ to fields $\mathbb{F}_{3^n}$ with surjective image $\mathrm{PSL}_3(\mathbb{F}_{3^n})$. It follows that this representation must have Zariski-dense image in $\mathrm{PSL}_3(\mathbb{F}_3[q, q^{-1}])$. Indeed, if the Zariski closure was a proper subgroup of $\mathrm{PSL}_3$, then it would have lower dimension, and then a Zariski-open subset of the algebraic curve $\mathrm{spec}(\mathbb{F}_3[q, q^{-1}])$ would have specialization of this dimension, and hence at most finitely many specializations can be of the form $\mathrm{PSL}_3(\mathbb{F}_{3^n})$.

Moreover, it was shown in \cite[Proposition 2.8]{chen2025finite} that specializations of this representation such that $F_2$ has surjective image to $\mathrm{PSL}_3$ give characteristic quotients of $F_2$. This was done by extending the map $F_2 \to \mathrm{PSL}_3(\mathbb{Z}[q, q^{-1}])$ to a map $\mathrm{Aut}(F_2) \to \mathrm{Aut}(\mathrm{PSL}_3(\mathbb{Z}[q, q^{-1}]))$, in a way that a finite index subgroup $\Gamma \leq \mathrm{Aut}(F_2)$ is mapped to the inner automorphisms $\mathrm{Inn}(\mathrm{PSL}_3(\mathbb{Z}[q, q^{-1}])) \simeq \mathrm{PSL}_3(\mathbb{Z}[q, q^{-1}])$. This $\Gamma$ is simply the index-$3$ subgroup of $\mathrm{Aut}^+(F_2)$ mapping to the index-$3$ subgroup $\mathrm{PSL}_3(\mathbb{Z}[q, q^{-1}]) \leq \mathrm{PGL}_3(\mathbb{Z}[q, q^{-1}]).$

We can now use \cite[Theorem B part ii]{larsen2004g2}. We have a Zariski-dense $\mathrm{SL}_3$-representation in characteristic $3$, so we obtain that the normal subgroup growth type of $\mathrm{Aut}(F_2)$ is $n^{\mathrm{log}(n)}$. Reading through the proof of this Theorem, this plethora of finite quotients is obtained using quotients of the form $\mathrm{SL}_3(\mathbb{F}_3[t] / (t^{3r}))$ for appropriate choices of $r$. This quotient has large center, and considering the family of its quotients by central subgroups gives us our needed finite quotients. In order to obtain the characteristic subgroup growth of $F_2$, it suffices to consider the same quotient, thinking of its center as a $G = \mathrm{Aut}(F_2) / \Gamma$-module, and divide by submodules of this $G$-module as in Corollary \ref{Lots of Submodules}. Note that we have used here that the image of the subgroup of inner automorphisms $F_2 \simeq \mathrm{Inn}(F_2) \lhd \mathrm{Aut}(F_2)$ also has Zariski-dense image and so will surject to these quotients, thereby yielding characteristic quotients of $F_2$.

\subsection{Characteristic growth of Lattices in semi-simple Lie groups}
Lattices in the Lie group $\mathrm{PSL}_2(\mathbb{R})$ are either virtually free groups or virtually surface groups. In both cases, these lattices are large. It follows from Theorem \ref{Characteristic Growth of Large Groups} that they have characteristic subgroup growth of type $n^{\mathrm{log}(n)}$, and in particular they have the same characteristic and normal subgroup growth types.

It follows from works of Agol \cite{agol2013virtual} and Wise \cite{wise2021structure} that lattices in $\mathrm{PSL}_2(\mathbb{C})$ are also large (see \cite[pp. 39]{aschenbrenner20123} for a detailed explanation of this fact) and so again these groups have the same characteristic and normal subgroup growth types, of type $n^{\mathrm{log}(n)}$.

A similar result can be shown for irreducible arithmetic lattices in high rank semi-simple Lie groups with the congruence subgroup property. In such a case, \cite[Theorem C]{larsen2004g2} computes the normal congruence subgroup growth, which is the same as the normal subgroup growth assuming the congruence subgroup property. There are two relevant cases.

Lattices in $G_2, F_4, E_8$ have polynomial normal subgroup growth, and also normal characteristic subgroup growth as the prinicipal congruence subgroups are all characteristic.

In every other case, the normal subgroup growth has type $n^{\frac{\mathrm{log}(n)}{\mathrm{log}(\mathrm{log}(n))^2}}.$ We use the method of proof in \cite{larsen2004g2}, showing that these finite quotients come from a specific product of quasi-simple groups, with large center. Such a quotient then yields plenty of other quotients by dividing by its central subgroups. Once again using Corollary \ref{Lots of Submodules} for the action of the finite outer automorphism group of the lattice (the finiteness following from Mostow rigidity), we get that the characteristic subgroup growth has the same type as the normal subgroup growth, $n^{\frac{\mathrm{log}(n)}{\mathrm{log}(\mathrm{log}(n))^2}}$.

The previous discussion leads to the following natural question
\begin{question}
Let $\Gamma$ be a finitely generated group, and assume that $\mathrm{Out}(\Gamma)$ is finite. Does $\Gamma$ have the same characteristic and normal subgroup growth types?
\end{question}
We could not find a simple proof of this fact, however, if it is true, then Mostow's rigidity theorem would imply that any irreducible lattice in a semi-simple Lie group has the same characteristic and normal subgroup growth types. Indeed, such a group either has finite outer automorphism group, or is virtually a lattice in $\mathrm{PSL}_2(\mathbb{R})$ and therefore large.
\bibliographystyle{alpha}
\bibliography{bibli}

\vspace{2cm}

\noindent{\textsc{DPMMS, Centre for Mathematical Sciences}}

\noindent{\textsc{Wilberforce Road, CB3 0WB,}}

\noindent{\textsc{Cambridge, UK}}

\textit{Email address:} \texttt{liamhanany@gmail.com}

\vspace{0.5cm}

\noindent{\textsc{Department of Mathematics,}}

\noindent{\textsc{Faculty of Mathematics and Computer Science,}}

\noindent{\textsc{Weizmann Institute of Science,}}

\noindent{\textsc{Rehovot, Israel}}

\textit{Email address:} \texttt{alex.lubotzky@mail.huji.ac.il}

\end{document}